\documentclass[12pt]{amsart}
\usepackage[left=80pt,right=80pt]{geometry}
\usepackage[usenames]{color}
\usepackage{amsmath}
\usepackage{amssymb}
\usepackage{amsthm}
\usepackage{enumitem}
\usepackage{float}
\usepackage{graphicx}
\usepackage{array}
\usepackage{tikz}
\usetikzlibrary{graphs,graphs.standard,calc}
\usetikzlibrary{matrix}
\usetikzlibrary{arrows}

\usepackage{hyperref,url}
\hypersetup{
	colorlinks=true,
  linkcolor=black,          
  citecolor=black,         
  filecolor=black,      
  urlcolor=black           
}

\newtheorem{prop}{Proposition}
\newtheorem{lemma}[prop]{Lemma}
\newtheorem{theorem}[prop]{Theorem}

\newtheorem{corollary}[prop]{Corollary}
\theoremstyle{definition}

\newtheorem{remark}[prop]{Remark}
\newtheorem{example}[prop]{Example}

\newcommand{\seqnum}[1]{\href{https://oeis.org/#1}{\rm \underline{#1}}}
\allowdisplaybreaks
\newcommand{\mylabel}[2]{#2\def\@currentlabel{#2}\label{#1}}

\begin{document}
\tikzset{mystyle/.style={matrix of nodes,
        nodes in empty cells,
        row 1/.style={nodes={draw=none}},
        row sep=-\pgflinewidth,
        column sep=-\pgflinewidth,
        nodes={draw,minimum width=1cm,minimum height=1cm,anchor=center}}}
\tikzset{mystyleb/.style={matrix of nodes,
        nodes in empty cells,
        row sep=-\pgflinewidth,
        column sep=-\pgflinewidth,
        nodes={draw,minimum width=1cm,minimum height=1cm,anchor=center}}}

\title{Further Results on Random Walk Labelings}

\author[SELA FRIED]{Sela Fried$^{\dagger}$}
\thanks{$^{\dagger}$ Department of Computer Science, Israel Academic College,
52275 Ramat Gan, Israel.
\\
\href{mailto:friedsela@gmail.com}{\tt friedsela@gmail.com}}
\author[TOUFIK MANSOUR]{Toufik Mansour$^{\sharp}$}
\thanks{$^{\sharp}$ Department of Mathematics, University of Haifa, 3498838 Haifa,
Israel.\\
\href{mailto:tmansour@univ.haifa.ac.il}{\tt tmansour@univ.haifa.ac.il}}

\maketitle

\begin{abstract}
Recently, we initiated the study of random walk labelings of graphs. These are graph labelings that are obtainable by performing a random walk on the graph, such that each vertex is labeled upon its first visit. In this work, we calculate the number of random walk labelings of several natural graph families: The wheel, fan, barbell, lollipop, tadpole, friendship, and snake graphs. Additionally, we prove several combinatorial identities that emerged during the calculations.

\bigskip


\smallskip

\noindent \textbf{Math.~Subj.~Class.:} 05C78, 05A10, 05A15, 05C81.
\end{abstract}

\section{Introduction}
Graph labeling is any assignment of labels, usually integers, to the edges or vertices of a graph. Often, the labelings are required to satisfy a certain property and the main interest is in establishing which graphs admit such labelings. For example, a graph $G$ with $m$ edges is called graceful if there is an injection $f$ from the vertices of $G$ to the set $\{0,1,\ldots,m\}$, such that, when each edge $xy$ is assigned the label $|f(x) -f(y)|$, the resulting edge labels are distinct (e.g., \cite[p.~5]{G}).

Although the literature on graph labelings is vast and many different kinds of labelings have been previously studied (see \cite{G} for a comprehensive literature survey), it seems that our work \cite{Fr} was the first to define and study random walk labelings. Let us recall their definition: Suppose $G$ is a connected and undirected graph with vertex set $V$ and let $n=|V|$. A \emph{random walk labeling} of $G$ is a labeling that is obtainable by performing the following process:
\begin{enumerate}
    \item Set $i=1$ and let $v\in V$. Label $v$ with $i$.
    \item As long as there are vertices that are not labeled, pick $w\in V$ that is adjacent to $v$ and replace $v$ with $w$. If $v$ is not labeled, increase $i$ by $1$ and label $v$ with $i$.
\end{enumerate}

Since every connected and undirected graph $G$ admits at least one random walk labeling, the main interest in this regard is the number of different random walk labelings that are possible for $G$, denoted by $\mathcal{L}(G)$. In \cite{Fr}, upon giving several examples, we calculated the number of random walk labelings of the grid graph and of the king's graph of size $2\times n$. In this work, we pursue the subject further and calculate the number of random walk labelings of several natural graph families, namely, the wheel, fan, barbell, lollipop, tadpole, friendship, and snake graphs. The first two graph families are formed by connecting an additional vertex to all the vertices of a given graph. For this construction, we obtain a general result (Theorem \ref{thm; 922}), from which we deduce the results for the wheel and fan graphs as examples. The next three graph families are special cases of graphs that are formed by joining two graphs with an edge, called bridge. It seems that, in this case, a general result in the spirit of Theorem \ref{thm; 922} would provide little gain, since it would heavily depend on the choice of the bridge vertices. Thus, we treat each of these three graph families separately (Theorems \ref{thm;5ah}, \ref{thm;66} and \ref{thm;7}). The last two graph families both result from connecting copies of the same cycle graph but differ in the way the cycles are connected. They are addressed by Theorems \ref{thm;kji} and \ref{thm;iop}. The work is concluded with proofs of several combinatorial identities that emerged during the calculations.

\section{Main results}

\subsection{Additional vertex connected to all vertices}

Let $m$ and $n$ be two natural numbers to be used throughout this work.

\begin{theorem}\label{thm; 922}
Let $G$ be a connected and undirected graph with vertex set $V$. Let $|V|=n$ and let $G'$ be the graph that is formed by connecting an additional vertex to all the vertices of $G$. For $1\leq k\leq n$, denote by $\mathcal{L}_k(G)$ the number of random walk labelings of $G$ that are disrupted immediately after the $k$th vertex has been labeled. Then
$$\mathcal{L}(G')=\sum_{k=0}^{n}(n-k)!\mathcal{L}_k(G),$$ where we set $\mathcal{L}_0(G) = 1$.
\end{theorem}
\begin{proof}
Let $1\leq k\leq n+1$ and assume that the label of the additional vertex is $k$. We distinguish between two cases:
\begin{enumerate}
    \item [$k=1$:] In this case every vertex of $G$ is reachable at any stage of the labeling and the number of such labelings is obviously given by $n!$.
    \item [$2\leq k\leq n+1$:] In this case, the first $k-1$ labels of $G'$ are within the graph $G$. These labels may be considered as a random walk labeling of $G$ that was disrupted immediately after the $k-1$th vertex has been labeled. The number of such random walk labelings is $\mathcal{L}_{k-1}(G)$. Now, once the additional vertex is labeled, the remaining $n - (k-1)$ vertices of $G$ may be labeled in arbitrary order, yielding $(n - (k-1))!$ possibilities.
\end{enumerate}
We conclude that
\begin{align*}
\mathcal{L}(G')&=n!+\sum_{k=2}^{n+1}(n-(k-1))!\mathcal{L}_{k-1}(G)=\sum_{k=0}^n(n-k)!\mathcal{L}_k(G). \qedhere
\end{align*}
\end{proof}

\begin{example}
\begin{enumerate}
\item [(a)] Recall that the wheel graph on $n+1$ vertices, denoted by $W_{n+1}$, is formed by connecting an additional vertex to all the vertices of the cycle graph $C_n$ (see, for example, \cite[p.~157]{H}. Figure \ref{fig:200} below visualizes $W_9$).

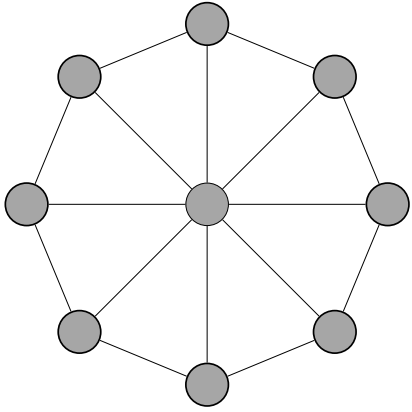
\begin{figure}[H]
\centering
\scalebox{0.8}{
\begin{tikzpicture}[shape = circle, node distance=4cm and 7cm,  nodes={ circle,fill=black!35}]
    \foreach \x in {1,...,8}{%
    \pgfmathparse{(\x-1)*360/8}
    \node[draw,circle,inner sep=0.25cm] (N-\x) at (\pgfmathresult:3cm) [thick] {};
  }
  \node[draw,circle,inner sep=0.25cm] (0,0) (mid) {};
  \path (N-1) edge[ultra thin,-] (N-2);
  \path (N-2) edge[ultra thin,-] (N-3);
  \path (N-3) edge[ultra thin,-] (N-4);
  \path (N-4) edge[ultra thin,-] (N-5);
  \path (N-5) edge[ultra thin,-] (N-6);
  \path (N-6) edge[ultra thin,-] (N-7);
  \path (N-7) edge[ultra thin,-] (N-8);
  \path (N-8) edge[ultra thin,-] (N-1);
  \path (N-1) edge[ultra thin,-] (mid);
  \path (N-2) edge[ultra thin,-] (mid);
  \path (N-3) edge[ultra thin,-] (mid);
  \path (N-4) edge[ultra thin,-] (mid);
  \path (N-5) edge[ultra thin,-] (mid);
  \path (N-6) edge[ultra thin,-] (mid);
  \path (N-7) edge[ultra thin,-] (mid);
  \path (N-8) edge[ultra thin,-] (mid);
\end{tikzpicture}}

\caption{The wheel graph $W_9$.}\label{fig:200}
\end{figure} \noindent We have $$\mathcal{L}(W_{n+1})=n\left((n-1)!-2^{n-2}+2^{n-1}\sum_{k=0}^{n-1}\frac{k!}{2^{k}}\right).$$ Indeed, reasoning as in \cite[Example 1(c)]{Fr}, it is easy to see that, for $1\leq k\leq n-1$, we have $\mathcal{L}_k(C_{n}) = n2^{k-1}$ and $\mathcal{L}_{n}(C_{n}) =\mathcal{L}(C_{n})= n2^{n-2}$. By Theorem \ref{thm; 922},
\begin{align*}
\mathcal{L}(W_{n+1})&=\sum_{k=0}^{n}(n-k)!\mathcal{L}_k(C_n)\\&=
n!+\sum_{k=1}^{n-1}(n-k)!n2^{k-1} + n2^{n-2} \\ &=n\left((n-1)!+\sum_{k=0}^{n-2}(n-1-k)!2^{k}+2^{n-2}\right)\\&=n\left((n-1)!-2^{n-2}+2^{n-1}\sum_{k=0}^{n-1}\frac{k!}{2^{k}}\right).
\end{align*}
In Theorem \ref{thm;ggd} in Section \ref{sec;af1}, we calculate the exponential generating function (egf) of the sequence $2^n\sum_{k=0}^n\frac{k!}{2^{k}}$, which is registered as \seqnum{A233449} in the On-Line Encyclopedia of Integer Sequences \cite{OL}. The expression for $\mathcal{L}(W_{n+1})$ seems to provide the first usage of this sequence.
\item [(b)] Recall that the fan graph on $n+1$ vertices, denoted by $F_{n+1}$, is formed by connecting an additional vertex to all the vertices of $P_n$ (see, for example, \cite[p.~781]{B}. Figure \ref{fig:205} below visualizes $F_6$).
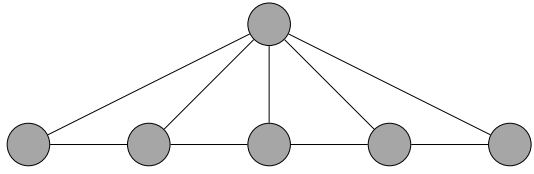
\begin{figure}[H]
\centering
\scalebox{0.8}{
\begin{tikzpicture}[shape = circle, node distance=4cm and 7cm,  nodes={ circle,fill=black!35}]
    \node[draw,circle,inner sep=0.25cm] at (0, 0)   (1) {};
    \node[draw,circle,inner sep=0.25cm] at (-4, -2)   (2) {};
    \node[draw,circle,inner sep=0.25cm] at (-2, -2)   (3) {};
    \node[draw,circle,inner sep=0.25cm] at (0, -2)   (4) {};
    \node[draw,circle,inner sep=0.25cm] at (2, -2)   (5) {};
    \node[draw,circle,inner sep=0.25cm] at (4, -2)   (6) {};
    \path (1) edge[ultra thin,-] (2);
    \path (1) edge[ultra thin,-] (3);
    \path (1) edge[ultra thin,-] (4);
    \path (1) edge[ultra thin,-] (5);
    \path (1) edge[ultra thin,-] (6);
    \path (2) edge[ultra thin,-] (3);
    \path (3) edge[ultra thin,-] (4);
    \path (4) edge[ultra thin,-] (5);
    \path (5) edge[ultra thin,-] (6);
\end{tikzpicture}}
\caption{The fan graph $F_6$.}\label{fig:205}
\end{figure} \noindent We have
$$\mathcal{L}(F_{n+1})=n!+\sum_{k=0}^{n-1}(n-k)!2^{k}.$$
To see that, we shall need to show that, for $1\leq k\leq n$, we have
\begin{equation}\label{eq;009}\mathcal{L}_k(P_n) = (n - k + 1)2^{k-1}.\end{equation}
Indeed, for $1\leq \ell\leq n$, there are $$\sum_{j=\max\left\{ 0,k-1-(n-\ell)\right\} }^{\text{min}\left\{ \ell-1,k-1\right\} }\binom{k-1}{j}$$ random walk labelings that begin at the $\ell$th vertex of $P_n$ (the index $j$ stands for the number of the vertices, left of the $\ell$th vertex, that are labeled by the random walk labeling). Thus,
\begin{align*}
\mathcal{L}_k(P_n) &=\sum_{\ell=1}^n\sum_{j=\max\left\{ 0,k-1-(n-\ell)\right\} }^{\text{min}\left\{ \ell-1,k-1\right\} }\binom{k-1}{j}\\
&=\sum_{\ell=1}^{n-k}\sum_{j=0}^{\ell-1}\binom{k-1}{j}+\sum_{\ell=n-k+1}^{n}\sum_{j=k-1-(n-\ell)}^{\ell-1}\binom{k-1}{j}\\
&=\sum_{\ell=1}^{n-k}\sum_{j=0}^{\ell-1}\binom{k-1}{j}+\sum_{\ell=1}^{k}\sum_{j=\ell-1}^{\ell-1+(n-k)}\binom{k-1}{j}\\
&=\sum_{\ell=1}^{n-k}\sum_{j=0}^{\ell-1}\binom{k-1}{j}+\sum_{\ell=1}^{k}\sum_{j=0}^{n-k}\binom{k-1}{j+\ell-1}\\
&=\sum_{\ell=1}^{n-k}\left(\sum_{j=0}^{\ell-1}\binom{k-1}{j}+\sum_{j=1}^{k}\binom{k-1}{j+\ell-1}\right)+\sum_{\ell=1}^{k}\binom{k-1}{\ell-1}\\
&=\sum_{\ell=1}^{n-k}2^{k-1}+2^{k-1}\\
&=(n-k+1)2^{k-1}.
\end{align*}
It follows that
\begin{align*}
\mathcal{L}(F_{n+1})&=\sum_{k=0}^{n}(n-k)!\mathcal{L}_k(P_n)\\&=n!+\sum_{k=1}^{n}(n-k+1)!2^{k-1}\\&=n!+\sum_{k=0}^{n-1}(n-k)!2^k.
\end{align*}
Notice that $$\sum_{k=0}^{n-1}(n-k)!2^{k}=\genfrac<>{0pt}{1}{n}{1},$$ where $\genfrac<>{0pt}{1}{n}{k}$ stands for the Eulerian number (e.g., \cite[pp.~253--258]{Gr} and \seqnum{A000295}). Furthermore, the expression for $\mathcal{L}_k(P_n)$
provides an additional combinatorial interpretation to \seqnum{A130128}.
\end{enumerate}
\end{example}

\subsection{Two graphs connected by a bridge}

In this section, we calculate the number of random walk labelings of graphs that are formed by connecting two graphs with an edge, called bridge.

Recall that the $n$-barbell graph, denoted by $B_n$, is the graph obtained by joining two complete graphs $K_n$ with a bridge (see, for example, \cite[p.~344]{G}. Figure \ref{fig:3011} below visualizes $B_8$). Generalizing, let us denote by $B_{m,n}$ the graph obtained by joining two complete graphs $K_m$ and $K_n$ with a bridge.

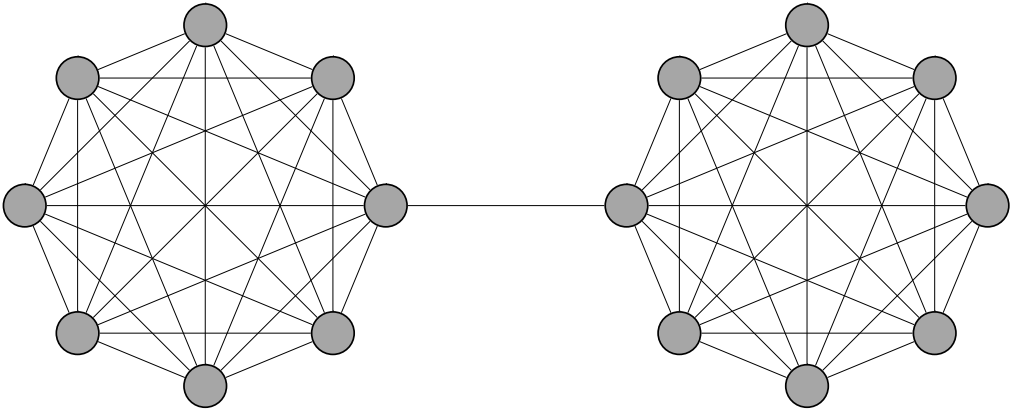
\begin{figure}[H]
\centering
\scalebox{0.8}{
\begin{tikzpicture}[shape = circle, node distance=4cm and 7cm,  nodes={ circle,fill=black!35}]
  \foreach \x in {1,...,8}{%
    \pgfmathparse{(\x-1)*360/8}
    \node[draw,circle,inner sep=0.25cm] (N-\x) at (\pgfmathresult:3cm) [thick] {};
  }
  \foreach \x in {1,...,8}{%
    \foreach \y in {\x,...,8}{%
        \path (N-\x) edge[ultra thin,-] (N-\y);
  }
  }
\begin{scope}[shift={(10,0)}]
  \foreach \x in {1,...,8}{%
    \pgfmathparse{(\x-1)*360/8}
    \node[draw,circle,inner sep=0.25cm] (O-\x) at (\pgfmathresult:3cm) [thick] {};
  }
  \foreach \x in {1,...,8}{%
    \foreach \y in {\x,...,8}{%
        \path (O-\x) edge[ultra thin,-] (O-\y);
  }
  }\end{scope}
  \path (N-1) edge[ultra thin,-] (O-5);
\end{tikzpicture}}
\caption{The barbell graph $B_{8}$.}\label{fig:3011}
\end{figure} \noindent

In the proof of the following theorem and also in several other places in this work, we shall make use of a well-known combinatorial identity (e.g., \cite[(1.48)]{Gou}), namely \begin{equation}\label{eq;0010}
\sum_{k=0}^n\binom{k+x}{r}=\binom{n+x+1}{r+1}-\binom{x}{r+1}, \end{equation}
that holds true for every real $x$ and every nonnegative integer $r$.

\begin{theorem}\label{thm;5ah}
We have $$\mathcal{L}(B_{m,n})=(m-1)!(n-1)!\left(\binom{m+n}{n+1}+\binom{m+n}{m+1}\right).$$
\end{theorem}
\begin{proof}
Suppose that the random walk labeling begins at the left complete graph, which we assume to be of size $m$, and suppose that the label of the left bridge vertex is $k$, where $1\leq k\leq m$. Let $\ell$ stand for the number of vertices that are labeled after the left bridge vertex is labeled and before the right bridge vertex is labeled. Thus, $0 \leq \ell \leq m-k$. The number of such random walk labelings is given by
\begin{align*}
&\sum_{k=1}^{m}(k-1)!\binom{m-1}{k-1}\sum_{\ell=0}^{m-k}\binom{m-k}{\ell}\ell!\binom{n-1+m-k-\ell}{m-1}(m-1)!(n-k-\ell)!\\
&=(m-1)!(n-1)!\sum_{k=1}^{m}\sum_{\ell=0}^{m-k}\binom{n-1+m-k-\ell}{n-1}\\&=(m-1)!(n-1)!\sum_{k=1}^{m}\binom{m-k+n}{n}\\&=(m-1)!(n-1)!\binom{m+n}{n+1}.
\end{align*} Interchanging the roles of $m$ and $n$ and adding both expressions, we obtain the claim.
\end{proof}

\begin{corollary}
We have $$\mathcal{L}(B_n)=2(n-1)!n!C_n,$$ where $C_n$ stands for the $n$th Catalan number.
\end{corollary}

Recall that the $(m,n)$-lollipop graph, denoted by $L_{m,n}$, is the graph obtained by joining a complete graph $K_m$ to a path graph $P_n$ with a bridge (see, for example, \cite[p.~344]{G}. Figure \ref{fig:300} below visualizes $L_{8,5}$).

\begin{figure}[H]
\centering
\scalebox{0.8}{
\begin{tikzpicture}[shape = circle, node distance=4cm and 7cm,  nodes={ circle,fill=black!35}]
  \foreach \x in {1,...,8}{%
    \pgfmathparse{(\x-1)*360/8}
    \node[draw,circle,inner sep=0.25cm] (N-\x) at (\pgfmathresult:3cm) [thick] {};
  }
  \foreach \x in {1,...,8}{%
    \foreach \y in {\x,...,8}{%
        \path (N-\x) edge[ultra thin,-] (N-\y);
  }
  }
\node[draw,circle,inner sep=0.25cm] (l1) at (5cm,0) [thick] {};
\node[draw,circle,inner sep=0.25cm] (l2) at (7cm,0) [thick] {};
\node[draw,circle,inner sep=0.25cm] (l3) at (9cm,0) [thick] {};
\node[draw,circle,inner sep=0.25cm] (l4) at (11cm,0) [thick] {};
\node[draw,circle,inner sep=0.25cm] (l5) at (13cm,0) [thick] {};
\path (N-1) edge[ultra thin,-] (l1);
\path (l1) edge[ultra thin,-] (l2);
\path (l2) edge[ultra thin,-] (l3);
\path (l3) edge[ultra thin,-] (l4);
\path (l4) edge[ultra thin,-] (l5);

\end{tikzpicture}}
\caption{The lollipop graph $L_{8,5}$.}\label{fig:300}
\end{figure}
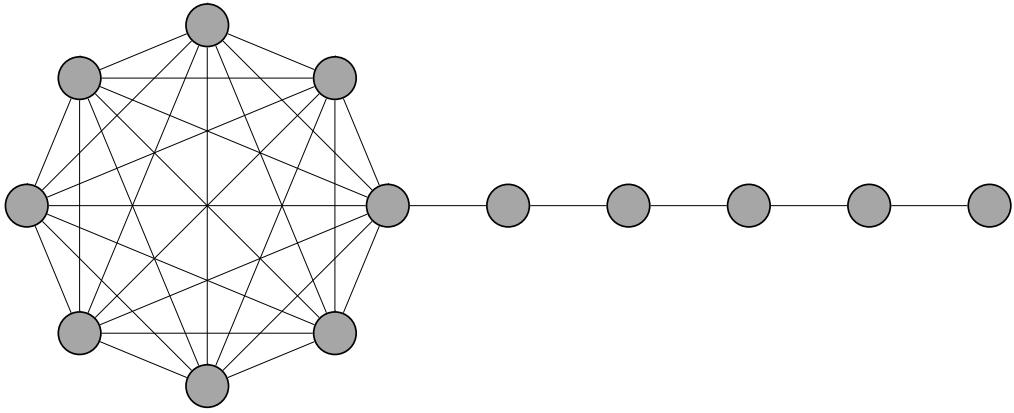 \noindent

We shall make use of the following combinatorial identity, for which we have not found a reference. Its proof is given in Section \ref{sec;af1} (Lemma \ref{lem;552}). For nonnegative integers $m$ and $n$ we have \begin{equation}\label{eq;kka}
2^{n}\sum_{k=0}^{n}\binom{k+m}{k}\frac{1}{2^{k}}=\sum_{k=0}^{n}\binom{m+1+n}{m+1+k}.
\end{equation}

\begin{theorem}\label{thm;66}
We have $$\mathcal{L}(L_{m,n})=(m-1)!\left(\binom{m+n}{n+1}+\sum_{k=0}^{n-1}\binom{m+n-1}{k+m}\right).$$
\end{theorem}
\begin{proof}
Let us refer to the bridge vertex belonging to the complete graph by $\alpha$. Let $1\leq k\leq m$ and assume that the random walk labeling begins at a vertex belonging to the complete graph and suppose that $\alpha$ has the label $k$. The number of such random walk labelings is given by $$
\binom{m-1}{k-1}(k-1)!(m-k)!\binom{m+n-k}{m-k}=(m-1)!\binom{m+n-k}{m-k}.$$ Summing over $k$, we conclude that the number of random walk labelings that begin at a vertex belonging to the complete graph is given by
$$(m-1)!\sum_{k=1}^m\binom{m+n-k}{m-k}=(m-1)!\binom{m+n}{n+1}.$$
Assume now that the random walk labeling begins at a vertex belonging to the path graph, such that $\alpha$ has the label $k$, where $2\leq k\leq n+1$. The number of such random walk labelings is given by
$$(m-1)!\binom{m+n-k}{m-1}\sum_{\ell=2}^{k}\binom{k-2}{\ell-2}=(m-1)!\binom{m+n-k}{m-1}2^{k-2}.$$ Summing over $k$ and using (\ref{eq;kka}), we conclude that the number of random walk labelings that start at a vertex belonging to the path graph is given by
\begin{align*}
\sum_{k=2}^{n+1}\binom{m+n-k}{m-1}2^{k-2}=2^{n-1}\sum_{k=0}^{n-1}\binom{k+m-1}{k}\frac{1}{2^k}&=\sum_{k=0}^{n-1}\binom{m+n-1}{k+m}. \qedhere
\end{align*}
\end{proof}

Recall that the $(m,n)$-tadpole graph, denoted by $T_{m,n}$, is the graph obtained by joining a cycle graph $C_m$ to a path graph $P_n$ with a bridge (see, for example, \cite[p.~18]{G}. Figure \ref{fig:4500} below visualizes $T_{8,5}$).

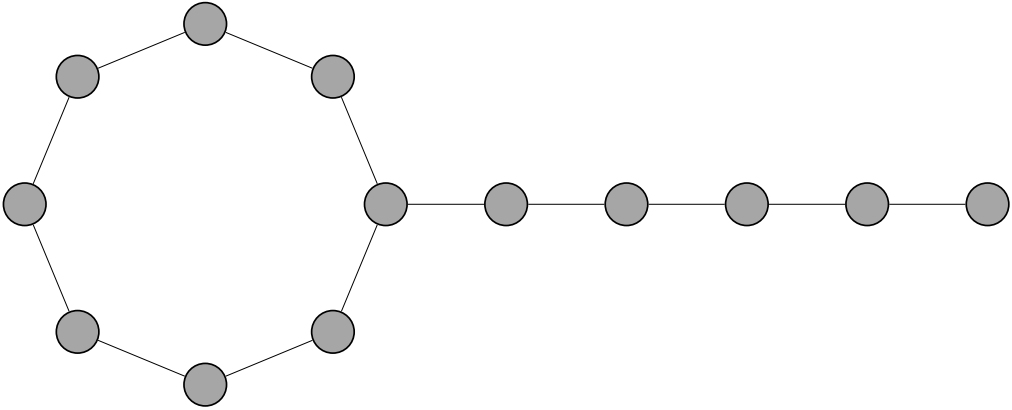
\begin{figure}[H]
\centering
\scalebox{0.8}{
\begin{tikzpicture}[shape = circle, node distance=4cm and 7cm,  nodes={ circle,fill=black!35}]
  \foreach \x in {1,...,8}{%
    \pgfmathparse{(\x-1)*360/8}
    \node[draw,circle,inner sep=0.25cm] (N-\x) at (\pgfmathresult:3cm) [thick] {};
  }

\path (N-1) edge[ultra thin,-] (N-2);
\path (N-2) edge[ultra thin,-] (N-3);
\path (N-3) edge[ultra thin,-] (N-4);
\path (N-4) edge[ultra thin,-] (N-5);
\path (N-5) edge[ultra thin,-] (N-6);
\path (N-6) edge[ultra thin,-] (N-7);
\path (N-7) edge[ultra thin,-] (N-8);
\path (N-8) edge[ultra thin,-] (N-1);

\node[draw,circle,inner sep=0.25cm] (l1) at (5cm,0) [thick] {};
\node[draw,circle,inner sep=0.25cm] (l2) at (7cm,0) [thick] {};
\node[draw,circle,inner sep=0.25cm] (l3) at (9cm,0) [thick] {};
\node[draw,circle,inner sep=0.25cm] (l4) at (11cm,0) [thick] {};
\node[draw,circle,inner sep=0.25cm] (l5) at (13cm,0) [thick] {};
\path (N-1) edge[ultra thin,-] (l1);
\path (l1) edge[ultra thin,-] (l2);
\path (l2) edge[ultra thin,-] (l3);
\path (l3) edge[ultra thin,-] (l4);
\path (l4) edge[ultra thin,-] (l5);

\end{tikzpicture}}
\caption{The tadpole graph $T_{8,5}$.}\label{fig:4500}
\end{figure} \noindent

\begin{theorem}\label{thm;7}
We have $$\mathcal{L}(T_{m,n})=2^{m-2}\left(\binom{m+n}{n+1}+\sum_{k=0}^{n-1}\binom{m+n-1}{k+m}\right).$$
\end{theorem}
\begin{proof}
Let us refer to the bridge vertex belonging to the cycle graph by $\alpha$ and let $1\leq k\leq m$. Assume that the random walk labeling begins at a vertex belonging to the cycle graph and suppose that $\alpha$ has the label $k$. If $k=1$ there are
$$2^{m-2}\binom{m+n-1}{n}$$ such random walk labelings and, if $k\geq 2$, there are $$2\cdot 2^{k-2}\binom{m+n-k}{n}2^{m-k-1}=2^{m-2}\binom{m+n-k}{n}$$ such random walk labelings. Suppose now that the random walk labelings begins at a vertex belonging to the path graph, such that the label of $\alpha$ is $k$, where $2\leq k\leq n+1$. There are $$2^{k-2}\binom{m+n-k}{m-1}2^{m-2}$$ such random walk labelings. Summing over $k$, we obtain the claim.
\end{proof}

\subsection{Connecting cycle graphs}

In this section, we consider two families of graphs that are obtained by connecting cycle graphs. The first is the one-point union of $m$ cycle graphs $C_n$, denoted by $C_n^{(m)}$, which, in the special case of $n=3$, coincides with the friendship graph $F_m$ (e.g, \cite[p.~18]{G}. Figure \ref{fig;133} below visualizes $F_4$).

\begin{figure}[H]
\centering
\scalebox{0.8}{
\begin{tikzpicture}[shape = circle, node distance=4cm and 7cm,  nodes={ circle,fill=black!35}]
\begin{scope}[rotate=22.5,transform shape]
  \foreach \x in {1,...,8}{%
    \pgfmathparse{(\x-1)*360/8}
    \node[draw,circle,inner sep=0.25cm] (N-\x) at (\pgfmathresult:3cm) [thick] {};
  }
  \node[draw,circle,inner sep=0.25cm] (0,0) (mid) {};
  \path (N-1) edge[ultra thin,-] (mid);
  \path (N-2) edge[ultra thin,-] (mid);
  \path (N-2) edge[ultra thin,-] (N-3);
  \path (N-4) edge[ultra thin,-] (mid);
  \path (N-3) edge[ultra thin,-] (mid);
  \path (N-4) edge[ultra thin,-] (N-5);
  \path (N-5) edge[ultra thin,-] (mid);
  \path (N-6) edge[ultra thin,-] (N-7);
  \path (N-6) edge[ultra thin,-] (mid);
  \path (N-7) edge[ultra thin,-] (mid);
  \path (N-8) edge[ultra thin,-] (N-1);
  \path (N-8) edge[ultra thin,-] (mid);

\end{scope}
\end{tikzpicture}}
\caption{The friendship graph $F_4$.}\label{fig;133}
\end{figure}
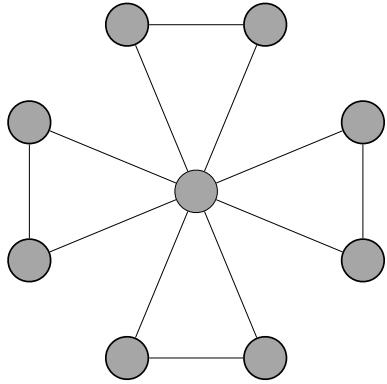 \noindent

\begin{theorem}\label{thm;kji}
We have $$\mathcal{L}(C_n^{(m)})=\left(2^{n-2}\right)^{m}\binom{(m-1)(n-1)}{\underbrace{n-1,\ldots,n-1}_{m-1\text{ times}}}\left(\binom{m(n-1)}{n-1}+m\binom{m(n-1)}{n-2}\right).$$
\end{theorem}

\begin{proof}
Clearly, there are \begin{equation}\label{eq;iuo}\left(2^{n-2}\right)^{m}\binom{m(n-1)}{\underbrace{n-1,\ldots, n-1}_{m\text{ times}}}\end{equation} random walk labelings that start at the middle vertex. Assume now that the random walk labeling starts at an inner vertex belonging to one of the $m$ $n$-cycles and that the middle vertex has the label $k$, where $2\leq k\leq n$. There are
$$2\cdot 2^{k-2}\binom{m(n-1)-(k-1)}{n-k,\underbrace{n-1,\ldots,n-1}_{m-1\text{ times}}}\left(2^{n-2}\right)^{m-1}2^{n-k-1}$$ such random walk labelings. Summing this over $k$, multiplying by $m$, and adding to (\ref{eq;iuo}), we obtain the claim.
\end{proof}

\begin{corollary}
    We have $$\mathcal{L}(F_m)=\frac{(4m-1)(2m)!}{2m-1}.$$
\end{corollary}

The second family of graphs we consider consists of graphs resulting from the following process: Begin with a path graph $P_{n+1}$. Now, replace each edge of the path with an edge of a cycle graph $C_m$. We denote the resulting graph by $S_{m, n}$. If $m=3$, then $S_{m, n}$ is known as a triangular snake (e.g., \cite[p.~19]{G}. Figure \ref{fig:206} below visualizes $S_{3,5}$).
\begin{figure}[H]
\centering
\scalebox{0.8}{
\begin{tikzpicture}[shape = circle, node distance=4cm and 7cm,  nodes={ circle,fill=black!35}]
\node[draw,circle,inner sep=0.25cm] at (-3, 0)   (10) {};
\node[draw,circle,inner sep=0.25cm] at (-1, 0)   (20) {};
\node[draw,circle,inner sep=0.25cm] at (1, 0)   (30) {};
\node[draw,circle,inner sep=0.25cm] at (3, 0)   (40) {};
\node[draw,circle,inner sep=0.25cm] at (5, 0)   (50) {};
    \node[draw,circle,inner sep=0.25cm] at (-4, -2)   (2) {};
    \node[draw,circle,inner sep=0.25cm] at (-2, -2)   (3) {};
    \node[draw,circle,inner sep=0.25cm] at (0, -2)   (4) {};
    \node[draw,circle,inner sep=0.25cm] at (2, -2)   (5) {};
    \node[draw,circle,inner sep=0.25cm] at (4, -2)   (6) {};
    \node[draw,circle,inner sep=0.25cm] at (6, -2)   (7) {};

\path (10) edge[ultra thin,-] (2);
\path (10) edge[ultra thin,-] (3);
\path (20) edge[ultra thin,-] (4);
\path (20) edge[ultra thin,-] (3);
\path (30) edge[ultra thin,-] (4);
\path (30) edge[ultra thin,-] (5);
\path (40) edge[ultra thin,-] (5);
\path (40) edge[ultra thin,-] (6);
\path (50) edge[ultra thin,-] (6);
\path (50) edge[ultra thin,-] (7);

\path (2) edge[ultra thin,-] (3);
\path (3) edge[ultra thin,-] (4);
\path (4) edge[ultra thin,-] (5);
\path (5) edge[ultra thin,-] (6);
\path (6) edge[ultra thin,-] (7);
\end{tikzpicture}}
\caption{The triangular snake $S_{3,5}$.}\label{fig:206}
\end{figure}
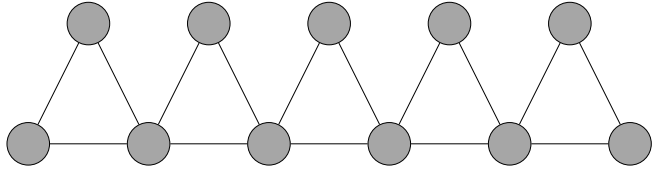 \noindent

\begin{theorem}\label{thm;iop}
Let $t=m-1$. Then
\begin{align*}
\mathcal{L}&(S_{m,n})\\&=\sum_{k=0}^{n-1}\sum_{j=2}^t\sum_{r=0}^{t-j-1}\sum_{s=0}^{kt}\binom{r+s}{s}\binom{nt-s-j-r}{kt-s,t-j-r,(n-1-k)t}2^{t-r-2}b_{m,k}b_{m,n-1-k}\\&+	
\sum_{k=0}^{n-1}\sum_{j=2}^t2^{j-1}\sum_{s=0}^{kt}\binom{t-j+s}{s}\binom{(n-1)t-s}{kt-s}b_{m,k}b_{m,n-1-k}\\&+	
\sum_{k=0}^{n}\binom{nt}{kt}b_{m,k}b_{m,n-k},	
\end{align*}
where $b_{m,n}$ stands for the number of random walk labelings of $S_{m,n}$, that start at one of the ends of the path graph and is recursively given by
$$b_{m,n} = \begin{cases}
1,&n=0;\\
2^{m-2},&n=1;\\
b_{m, n-1}\left(1+\sum_{j=2}^t\binom{nt+1-j}{m-j}2^{t-j}\right),& n>1.
\end{cases}$$
\end{theorem}

\begin{proof}
We distinguish between two cases:
\begin{enumerate}
\item The random walk labeling begins at the $k$th vertex of the path graph, where $1\leq k\leq n+1$. There are  $$\sum_{k=1}^{n+1}\binom{nt}{(k-1)t}b_{m,k-1}b_{m,n+1-k}=\sum_{k=0}^{n}\binom{nt}{kt}b_{m,k}b_{m,n-k}$$ such random walk labelings.
\item The random walk labeling begins at an inner vertex of the $k$th cycle, where $1\leq k\leq n$. Let us denote this cycle by $C$ and the left and right vertices of $C$, that lie on the path, by $\alpha$ and $\beta$, respectively. Notice that $\alpha$ and $\beta$ divide $S_{m,n}$ into three subgraphs, that are isomorphic to $S_{m,k-1}, C_n$ and $S_{m,n-k}$. Assume that $\alpha$ is labeled before $\beta$. Due to symmetry, we have a multiplicative factor of $2$. Let $j$ be the label of $\alpha$. Thus, $2\leq j\leq m-1$. At this stage, $j$ vertices of $C$ are labeled, yielding a multiplicative factor of $2^{j-2}$. Assume now that when $y$ is labeled, exactly $r$ additional vertices of $C$ and $s$ additional vertices of $S_{m, k-1}$ are already labeled, where $0\leq r\leq t-j$ and $0\leq s\leq (k-1)t$. This gives a multiplicative factor of $\binom{r+s}{s}$. Proceeding, we are now labeling $\beta$. At this stage, the vertices of the three subgraphs may be labeled. This gives a multiplicative factor of $$\binom{(k-1)t-s+t-j-r+(n-k)t}{(k-1)t-s,t-j-r,(n-k)t}=\binom{nt-s-j-r}{(k-1)t-s,t-j-r,(n-k)t}.$$ The random walk labelings of $S_{m,k-1}$ and $S_{m,n-k}$ necessarily begin at one of the ends of the path graph (relative to the subgraphs). This gives a multiplicative factor of $b_{m,k-1}b_{m,n-k}$. Finally, we consider the remaining arc of $C$. At this stage, $\beta$ is already labeled. Thus, if $0\leq r\leq t-j-1$, then we have a multiplicative factor of $2^{t-j-1-r}$. If $r=t-j$, we have a multiplicative factor of $1$.
\end{enumerate}
\end{proof}
\begin{corollary}
We have $$\mathcal{L}(S_{3,n}) = 2^{n}(n-1)!\left(\sum_{k=0}^{n-1}\frac{\binom{2n-1}{2k}}{\binom{n-1}{k}}+n\sum_{k=0}^{n}\frac{\binom{2n}{2k}}{\binom{n}{k}}\right).$$
\end{corollary}

\begin{proof}
It is not hard to see that $b_{3,n} = 2^nn!$, for $n\geq 0$.
It follows that
\begin{align*}
\mathcal{L}(S_{3,n})&=2^n\sum_{k=0}^{n-1}\sum_{s=0}^{2k}\binom{2(n-1)-s}{2k-s}k!(n-1-k)!+2^n	
\sum_{k=0}^{n}\binom{2n}{2k}k!(n-k)!\\&=2^{n}(n-1)!\sum_{k=0}^{n-1}\frac{1}{\binom{n-1}{k}}\sum_{s=0}^{2k}\binom{2(n-1)-s}{2k-s}+2^{n}n!\sum_{k=0}^{n}\frac{\binom{2n}{2k}}{\binom{n}{k}}\\&=2^{n}(n-1)!\left(\sum_{k=0}^{n-1}\frac{\binom{2n-1}{2k}}{\binom{n-1}{k}}+n\sum_{k=0}^{n}\frac{\binom{2n}{2k}}{\binom{n}{k}}\right)\qedhere
\end{align*}
\end{proof}
\begin{remark}
In \cite{Fr}, we obtained several results regarding the sequence \seqnum{A087547}. In particular, we have shown that $$\seqnum{A087547}(n)=(n-1)!\sum_{k=0}^{n-1}\frac{\binom{2n-1}{2k}}{\binom{n-1}{k}}.$$ In the next section (Theorem \ref{thm;gsa}), we shall prove another identity for this sequence.
\end{remark}

\subsection{Combinatorial identities}\label{sec;af1}

\begin{lemma}\label{lem;552}
For nonnegative integers $m$ and $n$ we have $$2^{n}\sum_{k=0}^{n}\binom{k+m}{k}\frac{1}{2^{k}}=\sum_{k=0}^{n}\binom{m+1+n}{m+1+k}.$$
\end{lemma}
\begin{proof}
We apply the snake-oil method (e.g., \cite[pp.~118--130]{W}). On one hand, we have
\begin{align*}	
\sum_{n\geq0}\sum_{m\geq0}\sum_{k=0}^{n}\frac{(2y)^{n}}{2^{k}}\binom{k+m}{k}x^{m}	&=\sum_{n\geq0}\sum_{k=0}^{n}\frac{(2y)^{n}}{2^{k}}\frac{1}{(1-x)^{k+1}}\\
&=\sum_{k\geq0}\frac{1}{2^{k}(1-x)^{k+1}}\sum_{n\geq k}(2y)^{n}\\
&=\sum_{k\geq0}\frac{1}{2^{k}(1-x)^{k+1}}\frac{(2y)^{k}}{1-2y}\\
&=\frac{1}{(1-2y)(1-x)}\sum_{k\geq0}\left(\frac{y}{1-x}\right)^{k}\\
&=\frac{1}{(1-2y)(1-x)}\frac{1}{1-\frac{y}{1-x}}\\
&=\frac{1}{(1-2y)(1-x-y)}.
\end{align*}
On the other hand, we have
\begin{align*}
\sum_{n\geq0}\sum_{m\geq0}\sum_{k=0}^{n}\binom{m+1+n}{m+1+k}y^{n}x^{m}	&=\sum_{k\geq0}\sum_{m\geq0}x^{m}\frac{1}{y^{m+1}}\sum_{n\geq k}\binom{m+1+n}{m+1+k}y^{m+1+n}\\
&=\sum_{k\geq0}\sum_{m\geq0}x^{m}\frac{1}{y^{m+1}}\frac{y^{m+1+k}}{(1-y)^{m+2+k}}\\
&=\frac{1}{(1-y)^{2}}\left(\sum_{m\geq0}\left(\frac{x}{1-y}\right)^{m}\right)\left(\sum_{k\geq0}\left(\frac{y}{1-y}\right)^{k}\right)\\
&=\frac{1}{(1-y)^{2}}\frac{1}{1-\frac{x}{1-y}}\frac{1}{1-\frac{y}{1-y}}\\
&=\frac{1}{(1-2y)(1-x-y)}.	\qedhere
\end{align*}
\end{proof}

\begin{theorem}\label{thm;gsa}
    We have \begin{equation}\label{eq;41}\seqnum{A087547}(n)=\frac{(n-1)!}{2}\sum_{k=0}^{n-1}\frac{\binom{2n}{2k+1}}{\binom{n-1}{k}}.\end{equation}
\end{theorem}
\begin{proof}
Denote the left- (resp.\ right) hand side of (\ref{eq;41}) by  $a_n$ (resp.\ $b_n$). We have
\begin{align*}
\sum_{n\geq1}\frac{b_n}{(n-1)!}x^n&=\frac{1}{2}\sum_{n\geq1}\sum_{k=0}^{n-1}\int_0^1n\binom{2n}{2k+1}t^k(1-t)^{n-1-k}x^ndt\\
&=\frac{1}{2}\int_0^1\sum_{n\geq1}\sum_{k\geq0}(n+k)\binom{2n+2k}{2k+1}t^k(1-t)^{n-1}x^{n+k}dt\\
&=-\int_0^1\frac{x(4t^2x^2-4tx^2+x^2-1)}{(4t^2x^2-4tx^2+(1-x)^2)^2}dt\\
&=\frac{x\left((1-x)\arctan\left(\frac{x}{\sqrt{1-2x}}\right)+\sqrt{1-2x}\right)}{(1-x)\left(\sqrt{1-2x}\right)^3}.
\end{align*}
It follows from the proof of Theorem 8 in \cite{Fr} that $a_n=b_n$.
\end{proof}

\begin{theorem}\label{thm;ggd}
    Let $f(x)$ be the egf of the sequence $$a_n = 2^n\sum_{k=0}^n\frac{k!}{2^{k}}.$$ Then $$f(x)=2e^{2x-2}\left(\int_0^1\frac{e^{2t}-e^{2(1-x)t}}{t}dt-\ln(1-x)\right)+\frac{1}{1-x}.$$
\end{theorem}

\begin{proof}
It is not hard to verify that $f(x)$ satisfies the homogeneous second order ordinary differential equation $$(x-1)f''(x)+2(-x+2)f'(x)+(-4)f(x)=0,$$ with initial conditions $f(0)=1$ and $f'(0)=3$. The solution of this equation is routine (e.g., \cite[ 14.1.2.108 on p.\ 531]{P}). Using Maple, we see that $$f(x)=2e^{2x-2}\left(\text{Ei}(2x-2)-\text{Ei}(-2)\right)+\frac{1}{1-x},$$  where $$\textnormal{Ei}(x) = \int_{-\infty}^x\frac{e^t}{t}dt$$ (e.g., \cite[p.\ 351]{Sp}).
\end{proof}


\begin{thebibliography}{9}

\bibitem{B}
Z.~R.~Bogdanowicz, Formulas for the number of spanning trees in a fan, {\it Appl.~Math.~Sci.} {\bf 2} (2008), 781--786.



\bibitem{Fr}
S.~Fried and T.~ Mansour, Graph labelings obtainable by random walks. To appear in Art Discrete Appl.\ Math., 2023. Available at \url{https://arxiv.org/pdf/2304.05728.pdf}.

\bibitem{G}
J.~A.~Gallian, A dynamic survey of graph labeling, {\it Electron.~J.~Combin.} (2018), \href{https://www.combinatorics.org/ds6}{Article DS6}.

\bibitem{Gou}
H.~W.~Gould, {\it Combinatorial Identities}, 1972.

\bibitem{Gr}
 R.~L.~Graham, D.~E.~Knuth, and O.~Patashnik, {\it Concrete Mathematics}, Addison-Wesley,
1994.

\bibitem{H}
F.~Harary, {\it Graph Theory}, Addison-Wesley, 1971.   	
	

\bibitem{OL}
N.~J.~A.~Sloane, The On-Line Encyclopedia of Integer Sequences, OEIS Foundation Inc., \url{https://oeis.org}.

\bibitem{P}
A.~D.~Polyanin and V.~F.~Zaitsev, {\it Handbook of Ordinary Differential Equations: Exact Solutions, Methods, and Problems}, CRC Press, 2017.

\bibitem{Sp}
 J.~Spanier and K.~B.~Oldham, {\it An Atlas of Functions}, Taylor \& Francis/Hemisphere, 1987.



\bibitem{W}
H.~S.~Wilf, {\it Generatingfunctionology}, CRC press, 2005.

\end{thebibliography}
\end{document}